\documentclass[12pts]{article}

\usepackage{times}
\usepackage{bm}
\usepackage{natbib}
\usepackage{amsmath,amsfonts,amssymb,amsthm}
\usepackage[figuresright]{rotating}
\usepackage{color}

\usepackage{geometry}
\usepackage{graphicx}

\usepackage{amsfonts}

\newtheorem{prop}{Proposition}[section]

\newcommand{\Ob}{{{\cal O}}}

\newcommand{\Ee}{{\rm E}}

\newcommand{\eps}{\varepsilon}

\newcommand{\dpg}[2]{\frac{\partial #1}{\partial #2}}
\newcommand{\ddpg}[2]{\frac{\partial^2 #1}{\partial #2^2}}

\newcommand{\si}{\sum_{i=1}^n}
\newcommand{\nn}{n^{-1}}

\newcommand{\tkl}{\hat \theta_{\kappa}}

\newcommand{\tki}{\hat \theta_{-i}}

\newcommand{\LCVa}{{\rm LCV}_a}

\newcommand{\pl}{{\rm pL}}

\newcommand{\iid}{i.i.d.}

\def\E{{\mathbb{E}}}

\def\Tr{\mbox{Trace}}

\def\limp{\overset{{p}}{\longrightarrow}}
\def\liml{\overset{{l}}{\longrightarrow}}

\title{Inference with penalized likelihood}

\author{Daniel Commenges $^{1,2}$
 \and J\'er\'emie Bureau $^{1,2}$
 \and Hein Putter $^{3}$
 }

\begin{document}





\maketitle
$^1$ INSERM, ISPED, Centre INSERM U-897-Epidemiologie-Biostatistique, Bordeaux,  F-33000\\
$^2$ Univ. Bordeaux, ISPED, Centre INSERM U-897-Epidemiologie-Biostatistique, Bordeaux,  F-33000, France\\
$^3$ Univ. Leiden, The Netherlands\\
\vspace{5mm}

\noindent {\bf ABSTRACT.}\\
This work studies the statistical properties of the maximum penalized likelihood approach in a semi-parametric framework. We recall the penalized likelihood approach for estimating a function and review some asymptotic results. We investigate the properties of two estimators of the variance of maximum penalized likelihood estimators: sandwich estimator and a Bayesian estimator. The coverage rates of confidence intervals based on these estimators are studied through a simulation study of survival data.
In a first simulation the coverage rates for the survival function and the hazard function are evaluated.
In a second simulation data are generated from a proportional hazard model with covariates.
The estimators of the variances of the regression coefficients are studied.
As for the survival and hazard functions, both sandwich and Bayesian estimators exhibit relatively good properties, but the Bayesian estimator seems to be more accurate. As for the regression coefficients, we focussed on the Bayesian estimator and found that it yielded good coverage rates.

\vspace{2mm}


\section{Introduction}

Penalized likelihood has been widely used for non or semi-parametric estimation of a function. It is also used in parametric models for variable selection. Here we are interested in  its use for estimating a hazard function as proposed by \cite{o1986statistical} and \cite{joly1998penalized}. More generally, this approach have been used for estimating transition intensities in multi-state models \cite{joly2002penalized}. It has the advantage that while making no parametric assumption on the hazard or intensity functions, it yields smooth estimates of these functions. Moreover, complex cases like interval-censored observations can be treated easily, at least conceptually. The main difficulty in the approach is the choice of the smoothing coefficient that can been done by cross-validation. There is also a theoretical difficulty in studying the properties of the maximum penalized likelihood estimators. There are mainly two candidates for estimating the variance of the maximum penalized likelihood estimators, inspired by M-estimators and Bayesian theory respectively. However precise theoretical results are lacking for justifying these estimators of the variance. The aim of this paper is to recall some facts about penalized likelihood and to study the two approaches for estimating the variances by simulation.

 Section 2 recalls the penalized likelihood approach for obtaining smooth non- or semi-parametric estimates of functions, together with the choice of the smoothing coefficient by approximate cross-validation. In section 3 we recall some asymptotic results and we investigate possible estimators of the variance of the penalized likelihood estimators, especially using a Bayesian approach and results from M-estimators theory. The case of survival data is developed in section 4. A simulation study is presented in section 5.  Section 6 concludes.

\section{The penalized likelihood and crossvalidation}\label{penlik}
\subsection{Penalized likelihood}

In the case where at least one of the parameters is a function, penalized likelihood can be applied for obtaining a smooth estimator \citep{o1986statistical}. The penalized loglikelihood is
$$\pl(\theta;\kappa)=L(\theta)-\kappa J(\theta),$$
where $L(\theta)$ is the loglikelihood and  $\kappa J(\theta)$ is a penalty term. For instance if the sample $\bar \Ob_n$ is constituted of $n$ \iid ~variables $Y_i$, the loglikelihood for $g^{\theta}$ is $\si \log g^{\theta}(Y_i)$. We consider the semi-parametric case where  $\theta=(\alpha,\beta)$ where $\alpha$ is a function and $\beta$ is a real parameter. In this case, the penalty term often depends only on the function $\alpha$. The most often used form of the penalty is  $J(\alpha)= \int [\alpha '' (t)]^2 ~dt$. The smoothing parameter $\kappa$ weighs the penalty term. The maximum penalized likelihood estimator (MPLE) is $\tkl$ which maximizes $\pl(\theta; \kappa)$.

The MPLE of the function is most often approximated on a spline basis so that the problem reduces to a parametric one with a large number of parameters. In several cases, a spline is indeed the exact solution \citep{wahba1983bayesian}, so that a finite number of real parameters have to be estimated. Even in that case, the model is still nonparametric because the location of the knots is driven by the observations. In general, the spline is used for approximating the solution. In theory, the number of knots can be arbitrarily large: the larger, the better the approximation of $\tkl$. In the penalized likelihood approach, the splines are used only for approximating the solution of the maximisation of $\pl(\theta;\kappa)$: the degree of smoothness is tuned by the weight $\kappa$ given to the penalty.

\subsection{The approximate crossvalidation criterion}

The choice of the smoothing coefficient $\kappa$ is crucial. It is desirable that this choice be data-driven and the most common approach is cross-validation.
We will focus on likelihood cross-validation which estimates the crossentropy of the estimator \citep{commenges2012universal}.
The leave-one-out cross-validation criterion may be computationally demanding since it is necessary to run the maximization algorithm $n$ times for finding the $\tki, i=1,\ldots,n$ . For this reason an approximate formula is very useful. \cite{Com07} gave the following formula:

\begin{equation}\label{LCVa} \LCVa(\kappa)=-n^{-1}L(\tkl)+\Tr(H^{-1}_{\pl}K),\end{equation}
where $K=\nn \si \hat v_i \hat d_i^T$ with $\hat v_i=\dpg {L_i(\theta)}{\theta}|_{\tkl}$  the individual score taken at the value of the penalized likelihood estimator, and $\hat d_i=\frac{1}{n-1}(\hat v_i+\kappa \dpg{J(\theta)}{\theta}|_{\tkl})$, and where  $H_{\pl}= -\ddpg {L(\theta)}{\theta}_{|\tkl}+\kappa \ddpg{J(\theta)}{\theta}_{|\tkl}$. The correction term is often interpreted as the equivalent number of parameters \citep{Com07}.

\section{Asymptotic distribution of the penalized likelihood estimators}\label{asympt-cond}
\subsection{Consistency}

\cite{cox1990asymptotic,cox1996penalized} present consistency results for MPLE in a non parametric framework. The problem is formalized using a normalized penalized loglikelihood. We define $\bar{pL_n} =  \bar L_n(\theta) - \lambda_n J(\theta)$ where $\bar L_n = \nn L(\theta)$ and $\lambda_n=\nn \kappa_n$.
For consistency of $\hat \theta_n$, the sequence $\lambda_n$ must tend toward zero in a certain range of speed. We note $\lambda_n=O(n^{-\alpha})$, and consider $0\le \alpha \le 1$. If $\lambda_n$ tends to 0 when $n\rightarrow\infty$ the bias of $\hat \theta_n$ goes to 0. The variance of $\hat \theta_n$ tends to zero whatever $\alpha$, $0\le \alpha \le 1$. Consistency obtains whatever $\alpha >0$ \citep{yu2002penalized}.

In the parametric case we can deduce the consistency of the MPLE from that of the MLE in the case of $\lambda_n = o(n)$.

\begin{prop}
Let $\Theta$ be a compact parameter space. We assume that the Maximum Likelihood Estimator (MLE) $\hat \theta^{ML}$ satisfies the conditions of the non penalized maximum likelihood estimator of the theorem 5.7 from \citep{van2000asymptotic} : there exists a fixed function which not depend on $n$ $L(\theta;x)=\E[L_n(\theta;x)]$, and a value $\theta_0 \in \Theta$ such that, $\forall \eps >0$,
\begin{equation}
\sup_{\theta\in \Theta} | \bar L_n(\theta;x) - L(\theta) | \limp 0 \quad \text{(uniform convergence)}
\end{equation}
\begin{equation}
\sup_{\theta : \| \theta - \theta_0 \| \leq \eps} L(\theta) < L(\theta_0)
\end{equation}
Then, if $\lambda_n = o(n)$, $\hat \theta^{pL}$ converges in probability to $\theta_0$.
\end{prop}
\begin{proof}
Let $\theta\in \Theta$, the parameter space defined as a compact.
\begin{equation}
\sup_{\theta \in \Theta}| \bar{pL_n} (\theta) - L (\theta)  | \leq \sup_{\theta \in \Theta} |\lambda_n J(\theta)  | + \sup_{\theta \in \Theta} | \bar L_n (\theta) - L (\theta) |
\end{equation}
if $\lambda_n = o(n)$ then,
\begin{equation}
\sup_{\theta \in \Theta} |\lambda_n J(\theta)  | \longrightarrow 0  \quad \text{as} \; n \rightarrow \infty.
\end{equation}
And by the classical non penalized MLE properties,
\begin{equation}
\sup_{\theta\in \Theta} | \bar L_n(\theta;x) - L(\theta) | \limp 0 \quad \text{as} \; n \rightarrow \infty
\end{equation}
The function $L(\theta)$ satisfies (MLE properties) the root unicity condition : let $\theta_0 \in \Theta$ such that
\begin{equation}
\sup_{\theta : \| \theta - \theta_0 \| \leq \eps} L(\theta) < L(\theta_0)
\end{equation}
\\
According to the theorem 5.7 of \citep{van2000asymptotic}, $\hat \theta^{pL} \limp \theta_0$ as $n\rightarrow \infty$.
\end{proof}

The variance increases while the bias decreases  with $\alpha$. There is an optimal rate for the MSE which depends on the degree of differentiability of the unknown function. It is $\nn$ if the degree of differentiability of the unknown function is infinite.
As we shall see in section \ref{Bayesian} there is a link between penalized likelihood and the Bayesian approach. The penalized log-likelihood can be interpreted as the logarithm of the numerator of the posterior distribution and the term $\kappa J(\theta)$ is then up to an additive constant the log of the prior. With this interpretation, since the prior is fixed, $\kappa$ does not depend on $n$, and thus $\lambda_n=O(\nn)$.

In spite of their theoretical interest, these results are not very useful in practice where the problem is to find the optimal value for a given $n$. The most often used method for doing this is crossvalidation. There are results in density estimation about the optimality of the crossvalidation choice: the estimator using the value of $\lambda_n$ chosen by crossvalidation has asymptotically the same rate of convergence than using the optimal value of $\lambda_n$ \citep{hall1992smoothed}.

One may wonder how the sequence of $\lambda_n$ chosen by crossvalidation behaves. We performed a simulation of survival data following a Weibull distribution and with samples of different sizes going from $100$ to $2000$ observations. The smoothing parameters $\kappa_n$ were estimated with the approximate cross validation criterion $LCV_a$. Figure \ref{kappa} shows the behavior of the sequence $\kappa_n$, $n^{-1/2}\kappa_n$ and $\lambda_n=\nn \kappa_n$ as a function of $n$. Apparently the sequence $\lambda_n$ goes to 0 as the number of observations grows and the rate seems to slightly slower than  $\nn$, the optimal rate for infinitely differentiable functions.
\begin{figure}[!h]
\centering
\includegraphics[scale=0.8]{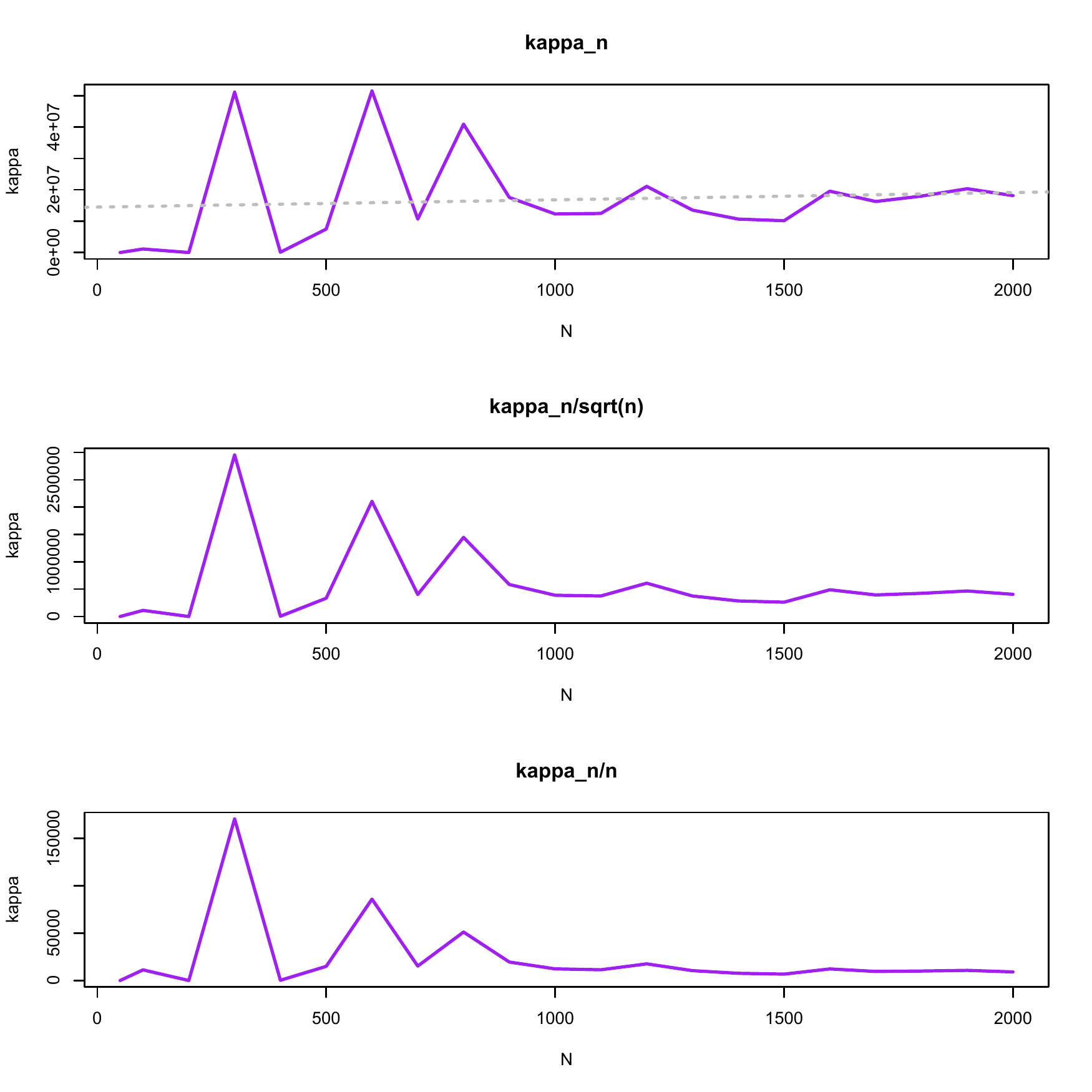}
\caption{Sequences $\kappa_n$, $n^{-1/2}\kappa_n$ and $\lambda_n$. The sequence of smoothing parameter $\kappa_n$ is computed using the $LCV_a$ crossvalidation criterion.}
\label{kappa}
\end{figure}

\subsection{Penalized likelihood as a posterior distribution}\label{Bayesian}

\cite{wahba1983bayesian} has made the link between penalized likelihood and the Bayesian approach. We must consider probability measures on a functional space. In this approach $\theta$ is considered as a random variable and the penalized loglikelihood appears as the numerator of the posterior density, where $\kappa J(\theta)$ is the log-prior density (up to an additive constant). Either $\kappa$ is known, or it is an hyper-parameter. It can be estimated from the data in an empirical Bayes spirit. Anyway, as a parameter it should not depend on $n$. Since $\lambda_n=\nn \kappa_n$, this choice is compatible with the optimal order of the sequence $\lambda_n$ in $\nn$ obtained by \cite{cox1996penalized} in the case of infinite differentiability.

\cite{o1988nonparametric} proposed an estimator based on a Bayesian view of the penalized likelihood. This leads to approximate the distribution of $\hat \theta$ by a Gaussian distribution  with  variance $V^{\text{Bayes}}_n(\hat \theta) = H^{-1}_{pL_n}(\hat \theta)$, where $H_{pL_n}(\hat \theta) = - \frac{\partial^2}{\partial\theta^2}pL_n(\hat \theta)$, the Hessian of minus the logarithm of the penalized likelihood. This estimator will be called the Bayesian estimator.

\subsection{Penalized likelihood estimators as M-estimators}

Since Huber's papers in the 1960's, M-estimation methods became important for asymptotic analysis and approximate inference. \cite{huber1964robust} introduced M-estimators and their asymptotic properties and they were an important part for the development of robust statistics. M-estimators were used to sudy the properties of maximum likelihood estimators in misspecified models. \cite{zeger1986longitudinal} used  M-estimator theory to justify the use of generalized estimating equations in the case of correlated observations.

Assuming $Y_i,\cdots,Y_n$ are independent and identically distributed, we can consider that the MPLE is a M-estimator \citep{van2000asymptotic}, that is an estimator which satisfies

\begin{equation}
\nn \si \phi(\theta,Y_i) = 0\quad \text{, with} \quad \phi(\theta,Y_i)=  - \dpg {L_i(\theta)}{\theta}  + \kappa \dpg{J(\theta)}{\theta}.
\end{equation}
The asymptotic distribution of $\hat\theta$ is the multivariate normal distribution ${\cal N}(\theta_0, V(\theta_0)$), where $\theta_0$ is the parameter value defined by $\Ee[\phi(\theta_0,Y_1)] = 0$ and where $V$ can be estimated by the (penalized) sandwich estimator :
\begin{equation}
\hat V(\theta_0) = H_{\pl}^{-1}(\theta_0) \bigg[\sum U_iU_i^T\bigg]H_{\pl}^{-1}(\theta_0),
\end{equation}
where the individual penalized score function are given by $U_i=(n-1)\hat d_i$, $\hat d_i=\frac{1}{n-1}(\hat v_i+\kappa \dpg{J(\theta)}{\theta}|_{\theta_0})$ and the non penalized score function $\hat v_i=\dpg {L_i(\theta)}{\theta}|_{\theta_0}$.

We also define the non-penalized sandwich estimator obtained by putting $\kappa=0$ in the formula. In that case, the non-penalized scores and Hessian are used and the formula is similar to the robust variance estimator of the variance of the MLE, except that the parameter value is taken at the maximum of the penalized likelihood (rather than the maximum of the likelihood).

M-estimation theory needs that the function $\phi$ does not depend on $i$ and $n$. \cite{chen2009asymptotic} used a M-estimator for functional estimation with a function depending on $n$ through a smoothing coefficient represented by a window width in a kernel method. \cite{cheng2010bootstrap} also derive results in which the $\phi(\theta,Y_i)$ may depend on $n$ and proves the consistency of a distribution obtained by bootstrap. Indeed bootstrap is another way of doing inference; however it is not feasible in complex problems, especially those in which the loglikelihood involves numerical integrals (as is the case for interval-censored observations). Actually, \cite{stefanski2002calculus} gave an extension of the M-estimators theory in which $\theta$ satisfies $\si \phi(\theta,Y_i) = s_n$, where $s_n/\sqrt{n}\limp 0$ as $n\rightarrow\infty$. In this particular case, the asymptotic property of $\hat\theta$ holds. This extension allows to cover a wider class of statistics whose $\phi$ function depends on $n$.

\section{Penalized likelihood inference for survival analysis}

We will focus on a proportional hazard model for survival data. We denote the independent survival times by $T_1,\cdots,T_n$. For $i=1,\ldots, n$ we observe $\tilde T_i=\min(T_i,C_i)$, where $C_i$ is a censoring variable. We use $i\in\mathcal{O}$ to denote the indexes of uncensored survival times. The proportional hazard model is given by the following expression for the hazard function:
\begin{equation}
h_i(t) = h_0 (t)\exp(X_i \beta)
\end{equation}
The aim of the study is to estimate the baseline hazard function $h_0(t)$ and the $p$ parameters vector $\beta$ of the covariates $X_i$.

The full likelihood function is then given by :

\begin{equation}
L(\beta,h_0) = - \sum_{i=1}^{n} H_0(\tilde T_i)\exp(X_i\beta) + \sum_{i\in\mathcal{O}}(\log h_0(\tilde T_i) + X_i\beta)
\end{equation}
where $H_0(t) = \int_{0}^{t} h_0(u)du$ is the cumulative baseline hazard function. We want to estimate the couple $(\beta,h_0)$ by using a penalized likelihood approach.

The penalized likelihood function is defined by
\begin{equation}
pL (\beta,h_0) = L(\beta,h_0) -  \kappa J(h_0) \quad, \;\kappa \geq 0.
\end{equation}
where $J(h_0)$ is the penalty function defined by $J(h_0) = \int h_0 ''(u)^2 du$.

Maximizing this penalized likelihood is a complex functional optimization problem. It is possible to approximate the solution on a spline basis \citep{joly1998penalized}:

\begin{equation}
\tilde h_0(t) = \sum_{k=1}^{m}\theta_k\psi_k (t)
\end{equation}
where $\psi_k (t)$ is the associated spline basis $(\psi_1,\cdots,\psi_m)$ of this finite dimensional space.

The likelihood expression of $L(\beta,h_0)$ then becomes
\begin{equation}
L(\beta,\theta) = -\sum_{k=1}^{m}\theta_k \sum_{i=1}^{n}\exp(X_i\beta)\Psi_k(\tilde T_i) + \sum_{i\in\mathcal{O}}(\log h_0(\tilde T_i) + X_i\beta)
\end{equation}
where the cumulative hazard function $H_0(t)$ is approximated with a proper spline basis $\Psi_k(t)\,,\,k=1,\cdots,m$;  such that $\Psi_k(t) = \int_0^t \psi_k(u)du$. Then, the cumulative baseline hazard function is approximated by
\begin{equation}
\tilde H_0(t) = \sum_{k=1}^{m}\theta_k\Psi_k (t)
\end{equation}
The penalty function becomes a function of the parameters vector $\bar\theta = (\theta_1,\cdots,\theta_m)^\top$ and can be expressed in a matrix form as (\cite{wahba1983bayesian})
\begin{equation}
J(\theta) = \bar\theta\,^\top \Omega \, \bar\theta
\end{equation}
where $\Omega$ is a symmetric $m\times m$ definite positive matrix with elements
\begin{equation}
\omega_{kr} = \int \psi_k ''(t)\psi_r ''(t)dt
\end{equation}
We are interested in the asymptotic properties of the vector $\hat\xi = (\hat\beta,\hat\theta)^\top$. Then, we deduce the asymptotic behavior of $(\tilde h_0(t) - h_0(t))$ and of $(\tilde S(t) - S(t))$ in order to get confidence intervals for the baseline hazard function and the survival function respectively.

Remark : why consider  $(\tilde h_0(t) - h_0(t))$ instead of  $(\widehat h_0(t) - h_0(t))$? The error $(\tilde h_0(t) - \widehat h_0(t))$ depends directly on the number of knots used in the spline approximation. Theoretically, one can set a sufficiently large number of knots in order to minimize the error.

\subsection{Variance estimation for the survival and hazard functions}
\noindent
The survival and hazard functions are approximated by linear combinations of cubic splines basis. The basis $\psi_k(t)$ and $\Psi_k(t)\,,\,k=1,\cdots,m$ then becomes a M-spline basis and I-spline basis respectively (\cite{rondeau2003maximum}) :
\begin{equation}
\label{survie}
\tilde S(t) = \exp\Big( - H(t) \Big) = \exp \Big( -\sum_{i=1}^{m} \theta_i I_i(t) \Big)
\end{equation}
\begin{equation}
\label{risque}
\tilde h(t) =\sum_{i=1}^{m} \theta_i M_i(t)
\end{equation}

For the hazard function, the variance is directly obtained from the expression (\ref{risque}) as a linear combination of the gaussian vector $\bar\theta = (\theta_1,\cdots,\theta_m)^\top$ :
\begin{equation}
\text{var}(\tilde h(t)) = \bar M(t)^\top  \text{var} (\bar\theta) \bar M(t)
\end{equation}
where $\bar M(t) = (M_1(t),\cdots,M_m(t))^\top$ is the M-spline basis used in the hazard function approximation.

Then we have the following asymptotic normal distribution as $n\rightarrow\infty$ :
\begin{equation}
\sqrt{n} \big(\tilde h(t,\bar\theta) - h(t,\bar\theta)\big) \liml N\Big(0,\bar M(t)^\top  \text{var} (\bar\theta) \bar M(t)\Big)
\end{equation}
The survival function case leads us to use the multivariate Delta Method (Appendix A) as the survival function is expressed by
\begin{equation}
\label{survie}
\tilde S(t) = \exp\Big( - H(t) \Big) = \exp \Big( -\sum_{i=1}^{m} \theta_i I_i(t) \Big)
\end{equation}
If we have the following asymptotic properties for the parameters vector $\bar\theta$ :

$ \bar\theta \limp \bar\theta_0$ and $\sqrt{n} (\bar\theta - \bar\theta_0) \liml N(0,\text{var}(\bar\theta))$ as n $\rightarrow\infty $, then, the Delta Method gives
\begin{equation}
\sqrt{n} \big(\tilde S(t,\bar\theta) -S(t,\bar\theta_0)\big) \liml N\Big(0,\nabla S(t,\bar\theta_0)^\top \text{var} (\bar\theta)\nabla S(t,\bar\theta_0)\Big)
\end{equation}
where
\begin{eqnarray*}
\nabla S(t,\bar\theta_0) &=& \Big(  \frac{\partial}{\partial \theta^{(1)}}S(t,\theta_0^{(1)}) , \cdots, \frac{\partial}{\partial \theta^{(m)}}S(t,\theta_0^{(m)}) \Big)
\\\\
&=& \Big( -I_1(t)\exp(-\sum_{i=1}^{m}\theta_i I_i(t)),\cdots, -I_m(t)\exp(-\sum_{i=1}^{m}\theta_i I_i(t)) \Big)
\end{eqnarray*}
The confidence interval of level $1-\alpha$ for these two functions are :
\begin{eqnarray*}
IC_{1-\alpha}\big( S(t,\bar\theta_0) \big) &=&  \tilde S(t,\bar\theta) \pm \Phi(1-\frac{\alpha}{2}) \bigg(\frac{\nabla S(t,\bar\theta_0)^\top \text{var} (\bar\theta)\nabla S(t,\bar\theta_0)}{n}   \bigg)^{1/2}
\\\\
IC_{1-\alpha}\big( h_0(t,\bar\theta_0) \big) &=&  \tilde h_0(t,\bar\theta) \pm \Phi(1-\frac{\alpha}{2}) \bigg(\frac{\bar M(t)^\top  \text{var} (\bar\theta) \bar M(t)}{n}   \bigg)^{1/2}
\end{eqnarray*}
where $\Phi$ is the quantile function of the standard normal distribution.

\section{Simulations}
We performed a simulation study to compare the statistical properties in terms of coverage rates of the different variance estimators. Two cases were considered: without covariate in a simple survival analysis and with covariates through a proportional hazard model.

\subsection{Inference for hazard and survival functions}\label{simulation1}

\subsubsection{Design}
Here we are interested in the coverage rates of survival or hazard functions estimators. The methodology is to generate a sample of observations following a known distribution, then compute the associated theoretical survival or hazard function and see if the expected (known) function belongs to the estimated confidence interval of the estimated function. The coverage rates have been computed as a mean coverage rates of $100$ confidence intervals taken at equally-spaced times between the smallest and highest observed times of event. Three variance estimators are compared: the Bayesian estimator, the non-penalized sandwich estimator, and the penalized sandwich estimator.

We have generated a right-censored data sample from the Weibull distribution W(a,b). We denote $a$ the shape parameter and $b$ the scale parameter. The density function of a Weibull distribution is given by :
\begin{equation}
f(x) = \frac{a}{b}\big(\frac{x}{b}\big)^{a-1}\exp\big\{ -\big(\frac{x}{b}\big)^a\big\} \quad, x>0
\end{equation}
The parameter values have been set at $a = 13$ and $b = 100$ in order to be close to real survival times. All observations are between 55 and 105. The censoring proportion was set at $20\%$ and uniformly distributed on the sample. The hazard function was approximated using cubic splines and the number of knots was set at $7$. The smoothing coefficient was chosen by the approximated cross-valdation criterion LCV$_a$. The coverage rates were computed as an average over $1000$ replicas. In order to see how each method is affected by sample size, the estimations were realized for three different sample sizes : $100$, $500$ and $1000$.

\subsubsection{Results}
Table \ref{tablesurv} (resp. Table \ref{tablehazard} ) show the mean coverage rates obtained with the three variance estimators for the survival function (resp. the hazard function). It appears that the Bayesian estimator gives good coverage rates whatever the sample size for both survival and  hazard functions. The non-penalized sandwich estimator seems not to work correctly. The penalized sandwich estimator works  correctly and seems to be more accurate for the hazard function than for the  survival function.

\begin{table}[htdp]
\begin{center}
\begin{tabular}{|c|c|c|c|}
\hline  Observations &  Bayes & Non Penalized Sandwich & Penalized Sandwich \\
\hline
n=100   &  $92\%$ & $98\%$ &  $93\%$ \\

\hline
n=500   & $93\%$ & $95\%$ &  $87\%$ \\

\hline
n=1000   &  $95\%$ & $93\%$ &  $89\%$ \\

\hline	
\end{tabular}
\end{center}
\caption{Coverage rate of confidence intervals for the survival function based on three estimators of the variance: Bayesian estimator, non-penalized and penalized sandwich estimators.}
\label{tablesurv}
\end{table}

\begin{table}[htdp]
\begin{center}
\begin{tabular}{|c|c|c|c|}
\hline  Observations &  Bayes & Non Penalized Sandwich & Penalized Sandwich \\
\hline
N=100   & $92\%$ & $58\%$ &  $98\%$ \\

\hline
N=500   &   $93\%$ & $51\%$ &  $97\%$ \\

\hline
N=1000   &  $93\%$ & $50\%$ &  $97\%$ \\

\hline	
\end{tabular}
\end{center}
\caption{Coverage rate of confidence intervals for the hazard function}
\label{tablehazard}
\end{table}

\subsection{Inference for regression coefficient in a proportional hazard model}

\subsubsection{Design}
The second simulation is a study of a proportional hazard model with covariates. Survival times have been generated from a Weibull distribution as described below (\cite{bender2005generating}). The proportional hazard model is given by
\begin{equation}
h_i(t) = h_0 (t)\exp(X_i \beta) \quad , \quad i = 1,\cdots, n.
\end{equation}
where $X$ is the vector of the $p$ covariates, $\beta$ the vector of the regression coefficients and $h_0 (t)$ is the baseline hazard function. Let $U$ be a random variable with the uniform distribution $U([0,1])$. The survival times $T$ of the proportional hazard model, following a Weibull distribution with parameters $\lambda$ (scale) and $k$ (shape) can be expressed as
\begin{equation}
T = \frac{1}{\lambda} [-\log(U)\exp(X_i \beta_i)]^{1/k}
\end{equation}
The simulations were made with one and two independent covariates. The regression coefficients have been set to $\beta_1=1$, $\beta_2=-1$. The number of survival times was $n=3000$. The Weibull distribution parameters were W(12,$1/0.01$). As for the first simulation of section section \ref{simulation1}, the parameter values of the Weibull distribution have been set so that the resulting distribution is close to that of real survival times. All observations of events were between $55$ and $105$. The values of covariates $X_1$ and $X_2$ were generated from uniform distributions: $X_1\sim U([0,1])$, $X_2\sim U([0,3])$. The censoring level was set at $20\%$. We focused on the Bayesian estimator. We generated $1000$ replicas of these experiments.

\subsubsection{Results}
Table \ref{tablebeta1} shows results with one covariate. Table \ref{tablebeta2} shows results for two covariates. In both cases we see that the bias of the penalized likelihood estimator is very small and that the coverage rates based on the Bayesian estimator of the variance are satisfactory.

\begin{table}[htdp]
\begin{center}
\begin{tabular}{|c|c|c|c|c|}
\hline  $\beta_i$ & $\widehat\beta_i$ & $\widehat\sigma_{Bayes}$ ($\widehat\beta_i$) & CI width (Bayes) & Rate (Bayes) \\
\hline
$\beta_1 = 1$ & 1.00036 & 0.05 & 0.20 &  $95\%$  \\
\hline	
\end{tabular}
\end{center}
\caption{Inference with the the Bayesian estimator for regression coefficients in a proportional hazard model: 1000 simulations with 2 covariates.}
\label{tablebeta1}
\end{table}%

\begin{table}[htdp]
\begin{center}
\begin{tabular}{|c|c|c|c|c|}
\hline  $\beta_i$ & $\widehat\beta_i$ & $\widehat\sigma_{Bayes}$ ($\widehat\beta_i$)  & CI width (Bayes) & Rate (Bayes) \\
\hline
$\beta_1 = 1$ & 1.003 & 0.06 &  0.24 &  $96\%$  \\
\hline	
$\beta_2 = -1$ & -0.969 & 0.08 &  0.34 & $91\%$ \\
\hline
\end{tabular}
\end{center}
\caption{Inference with the the Bayesian estimator for regression coefficients in a proportional hazard model: 1000 simulations with 2 covariates.}
\label{tablebeta2}
\end{table}

\section{Conclusion}

We have recalled facts about the penalized likelihood approach in its use for estimating an unknown function. We have focussed on the case of survival data. We have compared by simulation two possible estimators of the variance of the parameters, one based on the M-estimator theory, the other based on a Bayesian interpretation of penalized likelihood. We have studied coverage rates of confidence intervals for the hazard and the survival function. Both approaches work correctly, although the Bayesian estimator looks more accurate. For inference for regression coefficients we have focussed on the Bayesian estimator and found very good coverage property of the confidence intervals based on it.

\section{Appendix A : Multivariate Delta Method}
\noindent
Let $B$ be the estimate of a real parameter $\beta$. We assume that we have the following asymptotic properties for the estimator $B$, as $n \rightarrow \infty$ :
\begin{eqnarray*}
&B&\limp\beta \quad (\text{consitency}) \\
&\sqrt{n}& (B - \beta) \liml N(0,\Sigma) \quad (\text{asymptotic normality})
\end{eqnarray*}
where $\Sigma$ is a given covariance matrix.

We have the following approximation by using a first order Taylor expansion
\begin{equation}
g(B) \sim h(B) + \nabla h(B)^\top (B-\beta) + o(B)
\end{equation}
Then,
\begin{eqnarray*}
\text{var}(g(B)) &\sim& \text{var}(h(B) + \nabla h(B)^\top (B-\beta)) \\
			&=& \text{var}(h(B) + \nabla h(B)^\top B - \nabla h(B)^\top\beta)) \\
			&=& \text{var}(\nabla h(B)^\top B) \\
			&=& \nabla h(B)^\top \;\;\text{var}(B)\;\; \nabla h(B) \\
			&=&  \nabla h(B)^\top \;\; \frac{\Sigma}{n} \;\; \nabla h(B)
\end{eqnarray*}
Thus we have the following asymptotic normality
\begin{equation}
\sqrt{n} (g(B) - g(\beta)) \limp N\Big(0,\nabla h(B)^\top \;\;  \Sigma \;\; \nabla h(B)\Big)
\end{equation}

\bibliographystyle{chicago}

\bibliography{infpen}

\begin{thebibliography}{}

\bibitem[\protect\citeauthoryear{Bender, Augustin, and Blettner}{Bender
  et~al.}{2005}]{bender2005generating}
Bender, R., T.~Augustin, and M.~Blettner (2005).
\newblock Generating survival times to simulate cox proportional hazards
  models.
\newblock {\em Statistics in medicine\/}~{\em 24\/}(11), 1713--1723.

\bibitem[\protect\citeauthoryear{Chen and Zhang}{Chen and
  Zhang}{2009}]{chen2009asymptotic}
Chen, J. and L.~Zhang (2009).
\newblock Asymptotic properties of nonparametric m-estimation for mixing
  functional data.
\newblock {\em Journal of Statistical Planning and Inference\/}~{\em 139\/}(2),
  533--546.

\bibitem[\protect\citeauthoryear{Cheng and Huang}{Cheng and
  Huang}{2010}]{cheng2010bootstrap}
Cheng, G. and J.~Huang (2010).
\newblock Bootstrap consistency for general semiparametric m-estimation.
\newblock {\em The Annals of Statistics\/}~{\em 38\/}(5), 2884--2915.

\bibitem[\protect\citeauthoryear{Commenges, Joly, Gegout-Petit, and
  Liquet}{Commenges et~al.}{2007}]{Com07}
Commenges, D., P.~Joly, A.~Gegout-Petit, and B.~Liquet (2007).
\newblock Choice between semi-parametric estimators of markov and non-markov
  multi-state models from generally coarsened observations.
\newblock {\em Scandinavian Journal of Statistics\/}~{\em 34}, 33--52.

\bibitem[\protect\citeauthoryear{Commenges, Proust-Lima, Samieri, and
  Liquet}{Commenges et~al.}{2012}]{commenges2012universal}
Commenges, D., C.~Proust-Lima, C.~Samieri, and B.~Liquet (2012).
\newblock A universal approximate cross-validation criterion and its asymptotic
  distribution.
\newblock {\em Arxiv preprint arXiv:1206.1753\/}.

\bibitem[\protect\citeauthoryear{Cox and O'Sullivan}{Cox and
  O'Sullivan}{1990}]{cox1990asymptotic}
Cox, D.~D. and F.~O'Sullivan (1990).
\newblock Asymptotic analysis of penalized likelihood and related estimators.
\newblock {\em The Annals of Statistics\/}, 1676--1695.

\bibitem[\protect\citeauthoryear{Cox and O'Sullivan}{Cox and
  O'Sullivan}{1996}]{cox1996penalized}
Cox, D.~D. and F.~O'Sullivan (1996).
\newblock Penalized likelihood-type estimators for generalized nonparametric
  regression.
\newblock {\em Journal of multivariate analysis\/}~{\em 56\/}(2), 185--206.

\bibitem[\protect\citeauthoryear{Hall, Marron, and Park}{Hall
  et~al.}{1992}]{hall1992smoothed}
Hall, P., J.~Marron, and B.~U. Park (1992).
\newblock Smoothed cross-validation.
\newblock {\em Probability Theory and Related Fields\/}~{\em 92\/}(1), 1--20.

\bibitem[\protect\citeauthoryear{Huber}{Huber}{1964}]{huber1964robust}
Huber, P.~J. (1964).
\newblock Robust estimation of a location parameter.
\newblock {\em The Annals of Mathematical Statistics\/}~{\em 35\/}(1), 73--101.

\bibitem[\protect\citeauthoryear{Joly, Commenges, Helmer, and Letenneur}{Joly
  et~al.}{2002}]{joly2002penalized}
Joly, P., D.~Commenges, C.~Helmer, and L.~Letenneur (2002).
\newblock A penalized likelihood approach for an illness--death model with
  interval-censored data: application to age-specific incidence of dementia.
\newblock {\em Biostatistics\/}~{\em 3\/}(3), 433--443.

\bibitem[\protect\citeauthoryear{Joly, Commenges, and Letenneur}{Joly
  et~al.}{1998}]{joly1998penalized}
Joly, P., D.~Commenges, and L.~Letenneur (1998).
\newblock A penalized likelihood approach for arbitrarily censored and
  truncated data: application to age-specific incidence of dementia.
\newblock {\em Biometrics\/}, 185--194.

\bibitem[\protect\citeauthoryear{O'Sullivan}{O'Sullivan}{1986}]{o1986statistical}
O'Sullivan, F. (1986).
\newblock {A statistical perspective on ill-posed inverse problems}.
\newblock {\em Statistical Science\/}~{\em 1\/}(4), 502--518.

\bibitem[\protect\citeauthoryear{O'Sullivan}{O'Sullivan}{1988}]{o1988nonparametric}
O'Sullivan, F. (1988).
\newblock Nonparametric estimation of relative risk using splines and
  cross-validation.
\newblock {\em SIAM Journal on Scientific and Statistical Computing\/}~{\em
  9\/}(3), 531--542.

\bibitem[\protect\citeauthoryear{Rondeau, Commenges, and Joly}{Rondeau
  et~al.}{2003}]{rondeau2003maximum}
Rondeau, V., D.~Commenges, and P.~Joly (2003).
\newblock Maximum penalized likelihood estimation in a gamma-frailty model.
\newblock {\em Lifetime data analysis\/}~{\em 9\/}(2), 139--153.

\bibitem[\protect\citeauthoryear{Stefanski and Boos}{Stefanski and
  Boos}{2002}]{stefanski2002calculus}
Stefanski, L.~A. and D.~D. Boos (2002).
\newblock The calculus of m-estimation.
\newblock {\em The American Statistician\/}~{\em 56\/}(1), 29--38.

\bibitem[\protect\citeauthoryear{Van~der Vaart}{Van~der
  Vaart}{2000}]{van2000asymptotic}
Van~der Vaart, A. (2000).
\newblock {\em {Asymptotic Statistics}}.
\newblock Cambridge University Press.

\bibitem[\protect\citeauthoryear{Wahba}{Wahba}{1983}]{wahba1983bayesian}
Wahba, G. (1983).
\newblock Bayesian" confidence intervals" for the cross-validated smoothing
  spline.
\newblock {\em Journal of the Royal Statistical Society. Series B
  (Methodological)\/}, 133--150.

\bibitem[\protect\citeauthoryear{Yu and Ruppert}{Yu and
  Ruppert}{2002}]{yu2002penalized}
Yu, Y. and D.~Ruppert (2002).
\newblock Penalized spline estimation for partially linear single-index models.
\newblock {\em Journal of the American Statistical Association\/}~{\em
  97\/}(460).

\bibitem[\protect\citeauthoryear{Zeger and Liang}{Zeger and
  Liang}{1986}]{zeger1986longitudinal}
Zeger, S.~L. and K.-Y. Liang (1986).
\newblock Longitudinal data analysis for discrete and continuous outcomes.
\newblock {\em Biometrics\/}, 121--130.

\end{thebibliography}

\vspace{15mm}
\label{lastpage}
\end{document}